\documentclass[11pt,twoside]{amsart}
\usepackage{amsaddr}
\usepackage[english]{babel}

\usepackage{amssymb, palatino, xcolor}
\usepackage[all,cmtip]{xy}
\usepackage[colorinlistoftodos]{todonotes}
\usepackage{graphicx}
\usepackage[colorlinks=true, allcolors=blue]{hyperref}

\newtheorem{dfn}{Definition}[section]
\newtheorem{thm}[dfn]{Theorem}
\newtheorem{lem}[dfn]{Lemma}
\newtheorem{prop}[dfn]{Proposition}
\newtheorem{rem}[dfn]{Remark}
\newtheorem{example}[dfn]{Example}

\newenvironment{preuve}{{\em \bf Proof:}}{\hfill $\blacksquare$}

\def\P{\mathbb P}

\def\Q{\mathbb Q}
\def\R{\mathbb R}
\def\Z{\mathbb Z}
\def\H{\mathbb H}
\def\C{\mathbb C}

\definecolor{green}{rgb}{0,.6,0}

\title[Algebraic Toric Quasifolds]{Algebraic Toric Quasifolds}
\author[Fiammetta Battaglia and Elisa Prato]{Fiammetta Battaglia and Elisa Prato}
\address{Dipartimento di Matematica e Informatica "U. Dini"\\Universit\`a degli Studi di Firenze\\ Viale Morgagni 67/A\\ 50134 Firenze, ITALY\\fiammetta.battaglia@unifi.it\\ elisa.prato@unifi.it}
\thanks{The authors were partially supported by the 
PRIN Project ``Real and Complex Manifolds: Topology, Geometry and Holomorphic Dynamics" (MUR, Italy) and by GNSAGA (INdAM, Italy).}

\begin{document}

\begin{abstract} Symplectic and complex toric quasifolds are a generalization of toric manifolds and orbifolds  to the nonrational case. In this paper, we reframe these notions from the viewpoint of algebraic geometry.
\end{abstract}

\maketitle
\begin{small}

\medskip
\noindent \textbf{Mathematics~Subject~Classification:}
Primary: 14M25. Secondary: 58B20.
\end{small}

\section*{Introduction}
Toric quasifolds were first introduced by the second author in \cite{pcras, p} for the purpose of generalizing symplectic toric manifolds to the case of simple polytopes that are \textit{not necessarily rational}. Toric quasifolds are highly singular spaces that are locally modeled by open sets of $\R^{2n}$ modulo the action of countable groups; as such, they are a natural generalization of both toric manifolds and toric orbifolds. The work in the symplectic category
was shortly thereafter extended to the
complex category by both authors in \cite{cx}.
The aim of this paper is to reframe everything in the setting of algebraic geometry, following, as our main references and inspiration, the books by Fulton \cite{fulton} and Cox--Little--Schenck \cite{cox}. 

The foundation of the symplectic works \cite{pcras, p} consisted of establishing the right framework for the classical Delzant construction of symplectic toric manifolds \cite{delzant} to make sense in the nonrational case. Let us explain this in the language of fans. Let  $\Sigma$ be any simplicial fan in $\R^n$. If the fan is rational, the rays intersect a lattice $L$. In general, this will no longer be the case. However, one can always choose a set of ray generators $\{v_1,\ldots,v_d\}$ and consider their $\Z$--span, $Q$. This object is not necessarily a lattice; in the worst case scenario, it will be a dense subset of $\R^n$. It is an example of
\textit{quasilattice} (see Definition~\ref{quasilattice}). A more general way of proceeding is the following. Fix any quasilattice $Q$ in $\R^n$ first, and
then consider a simplicial fan $\Sigma$ whose ray generators all have nonempty intersection with $Q$.  Finally, choose a set of ray generators $\{v_1,\ldots,v_d\}$ belonging to $Q$. These choices define what 
we now call a \textit{fundamental triple}
$$(\Sigma, Q, \{v_1,\ldots,v_d\}).$$
Fundamental triples are the initial data required for the construction of all toric quasifolds: symplectic, complex, or algebraic. 

Once the right framework is established, it is quite surprising to see that, in all categories, many constructions have the same flavor as their rational counterparts. In our work \cite{cx}, for example, it had already emerged that complex toric quasifolds admit a canonical atlas with "affine" charts given by quotients of $\C^n$ modulo the linear action of countable subgroups $\Gamma$ of $(\C^*)^n$. This realization is necessarily the starting point for any algebraic formulation of the theory.

In this article, we build everything with the explicit aim of remaining as close as possible to the standard treatments of toric varieties \cite{fulton, cox}. In fact, we begin by carefully reviewing the classical constructions from our viewpoint, and we modify them only when needed. The key idea is that of replacing the lattice with the quasilattice $Q$, the primitive ray generators with the chosen ray generators
$v_1,\ldots,v_d\in Q$, and making the necessary changes along the way. This process yields, in a very natural way, the notion of \textit{affine toric quasifolds}, which turn out to be quotients of 
$\C^r\times(\C^*)^{n-r}$ modulo the action of the countable groups $\Gamma$ described above. As a special case, we get the \textit{quasitori} $(\C^*)^{n}/\Gamma$, which turn out to be isomorphic to the quotient $\C^n/Q$. 

The same idea applies to morphisms. We slightly modify the notion of cone morphism, only by requiring that the underlying linear map sends quasilattice to quasilattice. The induced affine toric morphism is then defined explicitly in terms of generalized Laurent monomials (see also \cite{laurent}). 
These monomials have real exponents, but in our quotients this is permissible (see Theorem~\ref{welldefined}). Verifying functoriality and equivariance with respect to quasitorus actions is a matter of direct computations (see Theorem~\ref{functoriality}, and Proposition~\ref{equivariance}).

Affine morphisms allow us to suitably glue affine toric quasifolds, yielding the notions of \textit{algebraic toric quasifolds} (see Theorem~\ref{gluing}) and their equivariant morphisms (see Theorem~\ref{toricmorphisms}). An interesting family of morphisms is provided by fan subdivisions and, in particular, blow--ups. We focus on the case of generalized Hirzebruch surfaces, $\H_a$, which form a one--parameter family of toric quasifolds containing the standard Hirzebruch surfaces $\H_n$ (see also \cite{hirze}). We show that they can be obtained by blowing--up generalized weighted projective space, in exactly the same way as in the classical case.

In fact, by the way that everything is constructed and generalized, it is straightforward to realize that, when the fan is rational, we obtain all of the standard theory back.

We remark that the nonsimplicial case was originally addressed by the first author in both the symplectic \cite{stratif-re} and the complex \cite{stratif-cx} setting and involves stratifications by symplectic and complex toric quasifolds, respectively. The algebraic point of view is work in progress and will appear in subsequent work.

Since the appearance of \cite{pcras,p,cx,stratif-re, stratif-cx}, nonrational toric geometry has been studied by a number of authors from different points of view. See, for example, \cite{bz1, rome, ratiu, HS, H, KLMV, IKP, boivin}.  Most of these treatments rely, in some way or another, on a choice of initial data that contain
a choice of ray generators and the
quasilattice given by their $\Z$--span. A discussion about how some of these different approaches are related can be found in \cite{whatis}.

The article is structured as follows. In Section~\ref{convex} we discuss prerequisites in convex geometry, in the rational and nonrational setting. In Section~\ref{affini} we introduce affine toric quasifolds and in Section~\ref{affinemorphism} their morphisms. In Section~\ref{toricquasifolds} we introduce algebraic toric quasifolds and in Section~\ref{morphisms} their morphisms and blow--ups. 

\section{Convex geometry}\label{convex}
\subsection{Cones and fans}
We recall a few facts about convex polyhedral cones; we refer to \cite{cox} and the references therein for a detailed treatment.

A \textit{convex polyhedral cone} in $\R^n$ is a set $$\rho=\left\{\sum_{i=1}^r t_i Y_i\mid t_i\in\R, t_i\geq0\right\}.$$
The vectors $Y_i\in\R^n$ are the \textit{generators} of the cone $\rho$. The cone $\{0\}$ is generated by the empty set. The \textit{dimension} of $\rho$ is the dimension of the linear subspace of $\R^n$ spanned by $\rho$.
The \textit{dual} of the convex polyhedral cone $\rho$ is given by
$$\rho^\vee=\{\alpha\in(\R^n)^*\mid \alpha(X)\geq0\quad\forall X\in\rho\}$$
and is itself a convex polyhedral cone. Moreover, $(\rho^\vee)^\vee=\rho$. 
Each non zero $\alpha\in(\R^n)^*$ defines a hyperplane $H_\alpha=\{X\in\R^n\;|\;\alpha(X)=0\}$ and a half-space $H^+_\alpha=\{X\in\R^n\;|\;\alpha(X)\geq0\}$. A hyperplane is said to be a \textit{supporting hyperplane} for a cone $\rho$ if $\rho\subset H_\alpha^+$. The intersection of $\rho$ with any supporting hyperplane, including the whole space, is called a \textit{face} of $\rho$. Faces that are strictly contained in $\rho$ are called \textit{proper}. Notice that a cone may contain a linear subspace, which is therefore contained in every face of the cone. For any face $\eta$, we can consider its \textit{relative interior}, $\text{relint}(\eta)$, which is the interior of $\eta$ in the linear subspace spanned by $\eta$ itself.   
A cone $\rho$ is said to be \textit{strongly convex} if it does not contain any line; this is equivalent to requiring that $\{0\}$ is a face of $\rho$.
In a strongly convex polyhedral cone the $1$--dimensional faces are all half--lines, the \textit{rays} of the cone. Given a strongly convex polyhedral cone $\rho$ with $r$ rays, and ray generators $v_1,\ldots,v_r$, we have that $\rho$ is generated by the $v_i$'s, namely $$\rho=\left\{\sum_{i=1}^rt_iv_i\mid t_i\in\R, t_i\geq0\right\}.$$
If, in addition, the vectors $v_1,\ldots,v_r$ are linearly independent, the cone $\rho$ has dimension $r$ and we will say that it
is \textit{simplicial}. In this case, for any subset $I_\eta=\{i_1,\ldots,i_h\}$ of $\{1,\ldots,r\}$, the corresponding vectors $v_{i_1},\ldots,v_{i_h}$ generate an $h$--dimensional face $\eta$ of the cone $\rho$, which is itself a simplicial cone, having dimension $h$. A maximal simplicial cone is a simplicial cone of dimension $n$.
Consider an $r$--dimensional simplicial cone $\rho$ as above and a basis $\{v_1,\ldots,v_r,v_{r+1},\ldots,v_n\}$ that completes the linearly independent set $\{v_1,\ldots,v_r\}$. If $\{\alpha_1,\ldots,\alpha_n\}$ is its dual basis, then the dual cone $\rho^{\vee}$ can be described, relatively to this basis, as the cone generated by $\alpha_i$, with $1\leq i\leq r$ and $\pm \alpha_i$ with $r+1\leq i\leq n$.
In particular, notice that when $\rho$ is maximal, $\rho^{\vee}$ is generated by 
$\{\alpha_1,\ldots,\alpha_n\}$. Thus we have that, for any $r$--dimensional simplicial cone, its dual is an $n$--dimensional cone that is strongly convex only when $\rho$ is maximal.

Let $L$ be the lattice in $\R^n$ that is generated by vectors $v_1,\ldots,v_n\in\R^n$:
$$L=\left\{\sum_{i=1}^{n}n_iv_i\mid n_i\in\Z\right\}. $$
The dual lattice $\text{Hom}(L,\Z)$ is isomorphic to
$$L^{\vee}=\{\alpha\in(\R^n)^*\;|\;\alpha(X)\in\Z\quad\forall\quad X\in L\}.$$ Notice that the dual basis $\alpha_1,\ldots,\alpha_n$ to $v_1,\ldots,v_n$ is a basis of $L^{\vee}$.

A \textit{fan} $\Sigma$ in $\R^n$ is a finite collection of strongly convex polyhedral cones in $\R^n$, such that each face of a cone in $\Sigma$ belongs to $\Sigma$ and the intersection of any two cones of $\Sigma$ is a face of each. The one--dimensional cones are called \textit{rays} of the fan.

A fan is \textit{simplicial} if its cones are all simplicial.

The \textit{support} $|\Sigma|$ of a fan $\Sigma$ is the union in $\R^n$ of its cones. From now on, we will consider fans
that have convex support of full dimension; this implies that $|\Sigma|$ is the union of the maximal cones of $\Sigma$. 
A notable example is the case of a \textit{complete} fan, 
namely a fan whose support is the whole space $\R^n$ like, for instance, the normal fan 
of a convex polytope. 

\subsection{Rational and smooth cones and fans}
Given a lattice $L$, a strongly convex polyhedral cone $\rho$ is \textit{rational} with respect to $L$, or in $L$, if each of its rays has a nonempty intersection with $L$. Thus, an $r$--dimensional rational cone $\rho$ determines a set of $r$ canonical generators. In fact, the intersection of each ray with the given lattice $L$ is nonempty and a discrete subset of the ray; therefore, the shortest vector in this intersection is well defined and is called \textit{the primitive generator} of the ray. 
A fan is \textit{rational} with respect to $L$, or in $L$, if every cone of the fan is rational in $L$.

A simplicial cone $\rho$ that is rational in $L$ is said to be \textit{smooth} if its canonical set of generators can be extended to a basis of $L$. A simplicial fan that is rational in $L$ is said to be \textit{smooth} if each of its cones is smooth.  

\subsection{Nonrational cones and fans} 
\noindent
In this subsection, we build the
framework for toric geometry in the nonrational setting.
We begin by recalling from \cite{pcras, p} the following
\begin{dfn}[Quasilattice]\label{quasilattice}
\rm{A \textit{quasilattice} $Q$ in $\R^n$ is the $\Z$--span of a set of generators of $\R^n$.}
\end{dfn}
A quasilattice in $\R^n$ is therefore a free $\Z$--module of rank greater or equal to $n$. The key feature that distinguishes a true lattice from a quasilattice of rank strictly greater than $n$ is of topological nature: a quasilattice in $\R^n$ has rank equal to $n$ if and only if it is a discrete subset of $\R^n$. 
A basic, though important, example of quasilattice is $Q=\Z+a\Z$, with $a$ an irrational number; it is dense in $\R$. 
\begin{dfn}[Quasirational cone]
\rm{A strongly convex polyhedral cone $\delta$ in $\R^n$ is \textit{quasirational} with respect to a quasilattice $Q$, or in $Q$, if each of its rays has nonempty intersection with $Q$.}
\end{dfn}
\begin{rem}\label{duali}
 \rm{Notice that, unlike the rational case,
 here we no longer have a canonical choice of generators, since there might be a ray whose intersection with $Q$ is not a discrete subset of the ray. Hence, the notion of primitive ray generator does not make sense.
 There are other significant changes with respect to the rational case. To begin with, Gordan's Lemma does not hold. Consider, for example, the one--dimensional cone 
$\sigma=\R_{\geq 0}$ in $\R$. Its intersection with the quasilattice $\Z+\sqrt{2}\,\Z$ is a semigroup that is dense in $\sigma$, whilst if it were finitely generated as a semigroup, it would be discrete. Another remarkable difference is that $$Q^\vee=\{\alpha\in(\R^n)^*\mid\alpha(X)\in\Z,\;\forall\;X\in Q\}$$ in general is a lattice in $(\R^n)^*$ whose rank is lower than $n$. In particular, when $Q$ is dense in $\R^n$, $Q^\vee=0$ . On the other hand, the rank of the free $\Z$--module $\text{Hom}(Q,\Z)$ is equal to $\text{rank}(Q)\geq n$, therefore in general  $\text{Hom}(Q,\Z)$ can no longer be identified with  $Q^\vee$.}
\end{rem}
\begin{dfn}[Quasirational fan]
 \rm{
A fan $\Sigma$ in $\R^n$ is \textit{quasirational} with respect to a quasilattice $Q$, or in $Q$, if every cone of $\Sigma$ is quasirational in $Q$.
}
\end{dfn}
\begin{rem}
\rm{It is very important to notice that a cone, or a fan, is always quasirational with respect to some quasilattice $Q$. It suffices to take the quasilattice that is generated by a choice of ray generators. So, unlike rationality,  quasirationality, perse, is not restrictive at all.}
\end{rem}
\begin{rem}\label{base}
\rm{Let $\delta$ be an $r$--dimensional simplicial cone in $\R^n$ that is quasirational in $Q$. Then, since $Q$ is generated by a set of spanning vectors of $\R^n$, by a standard linear algebra argument, we can always choose $v_1^\delta,\ldots,v_n^\delta$ in $Q$ which form a basis of $\R^n$ and $I_\delta=\{i_1,\ldots,i_r\}\subseteq \{1,\ldots,n\}$ such that $v_{i_1}^\delta,\ldots,v_{i_r}^\delta$ are ray generators of $\delta$. Notice that this implies that we can always construct a simplicial maximal cone $\sigma$ that is quasirational in $Q$ such that $\delta$ is a face of $\sigma$. }
\end{rem}
We adapt to this context the notion of fundamental triple.
\begin{dfn}[Cone fundamental triple]
\rm{
Let $Q\subset\R^n$ be a quasilattice. For any simplicial cone $\delta$ in $\R^n$ that is
quasirational in $Q$, a \textit{fundamental triple} is given by $$(\delta, Q, \{v_1^\delta,\ldots,v_n^\delta\})$$ where $v_1^\delta,\ldots,v_n^\delta$ are vectors chosen as in Remark \ref{base} above.}
\end{dfn}
Notice that there is a lattice 
which is
naturally associated with such a triple. It is the lattice $L_\delta$ generated by $v_1^\delta,\ldots,v_n^\delta$. With respect to this lattice, $\delta$ is a smooth cone in the classical sense.
\begin{rem}[Induced triple on a cone face]
\rm{
Consider a cone $\delta$ with fundamental triple 
$(\delta, Q, \{v^\delta_1,\ldots,v^\delta_n\}).$ Whenever we have a face $\rho\subset\delta$, we will consider the naturally induced triple
$(\rho, Q, \{v^\delta_1,\ldots,v^\delta_n\})$
for $\rho$.
Notice that $I_\rho\subseteq I_\delta$ and $L_\rho=L_\delta$.}
\end{rem}
\begin{dfn}[Fan fundamental triple]\label{triple}
\rm{
Let $\Sigma$ be a simplicial fan in $\R^n$ that is quasirational in $Q$. A \textit{fundamental triple} for
$\Sigma$ is given by
$$(\Sigma, Q, \{v_1,\ldots,v_d\}),$$ where the vectors $v_1,\ldots,v_d$ belong to $Q$ and generate the $d$ rays of $\Sigma$.}
\end{dfn}

Notice that for the maximal cones of the fan, the corresponding triples are uniquely determined by the fan triple. The lower dimensional cones may be faces of different
maximal cones and, as such, will be endowed with different induced triples.
A fan triple yields a collection of cone triples. If there is a unique lattice associated with each of them, then the fan is smooth with respect to this lattice in the classical sense.

Classically, given a lattice $L$ in $\R^n$,
and a rational simplicial fan $\Sigma$, there is an underlying fundamental triple which is given by 
$$(\Sigma, L, \{v_1,\ldots,v_d\}),$$
where the vectors $v_1,\ldots,v_d$ are the primitive ray generators for $\Sigma$.

\subsection{Cone and fan morphisms}
The notions of cone and fan morphism have to be adapted to the nonrational setting in the following natural way.
Consider a cone $\delta_1\subset\R^n$ that
is quasirational in a quasilattice $Q_1\subset \R^n$, and a cone
$\delta_2\subset\R^m$ that
is quasirational in a quasilattice $Q_2\subset \R^m$.
\begin{dfn}[Cone morphism]
\rm{ 
A \textit{cone morphism} $\delta_1\rightarrow \delta_2$
is the datum of a linear mapping $f\colon\R^n\rightarrow\R^m$  such that
$f(Q_1)\subset Q_2$ and
$f(\delta_1)\subset\delta_2$.
}
\end{dfn}
The identity on $\R^n$ defines 
a morphism of any cone
 $\delta\subset \R^n$ into itself
which we call the \textit{identity morphism}.
Two cone morphisms can be naturally
composed by taking the composition of the corresponding linear mappings. Moreover, we have
\begin{dfn}[Cone isomorphism]
\rm{A cone \textit{isomorphism}  $\delta_1\rightarrow \delta_2$ is the datum of a linear isomorphism $f$ of $\R^n$ such that 
$f(Q_1)= Q_2$ and $f(\delta_1)=\delta_2$.
The two cones are then said to be \textit{isomorphic}.}
\end{dfn}

Let $\Sigma_1\subset \R^n$ and $\Sigma_2\subset \R^m$ be fans that are quasirational, respectively, in quasilattices $Q_1\subset \R^n$ and $Q_2\subset \R^m$.
\begin{dfn}[Fan morphism]
\rm{A \textit{fan morphism} 
$\Sigma_1\rightarrow\Sigma_2$
is the datum of a linear mapping $f\colon\R^n\rightarrow\R^m$ with 
$f(Q_1)\subset Q_2$, and such that,
for each cone $\delta_1$ in $\Sigma_1$, there is a cone $\delta_2$ in $\Sigma_2$, with the property that $f(\delta_1)\subset \delta_2$.
}
 \end{dfn} 

As in the case of cones, the identity 
of $\R^n$ induces a morphism
of any  
$\Sigma\subset\R^n$,
into itself, called the \textit{identity morphism}, and
two fan morphisms can be composed by
taking the composition of the
corresponding linear maps.
We also have 
\begin{dfn}[Fan isomorphism]
 \rm{A \textit{fan isomorphism} $\Sigma_1\rightarrow\Sigma_2$
 is the datum of an isomorphism $f$ of $\R^n$ with 
$f(Q_1)= Q_2$, and such that,
for each cone $\delta_1$ in $\Sigma_1$, there is a cone $\delta_2$ in $\Sigma_2$, with the property that $f(\delta_1)= \delta_2$. The two fans are then said
 to be \textit{isomorphic}.}
\end{dfn}

\section[quasi affini]{Affine toric quasifolds}\label{affini}
In this section, we introduce affine toric quasifolds. We do so in stages, by first reviewing, from our perspective, what happens in the rational case.
\subsection{Smooth affine toric varieties}
We recall the classical construction of affine toric varieties corresponding to smooth cones. 

To begin with, let $\sigma$ be a maximal simplicial cone in $\R^n$ and let $v_1^\sigma,\ldots,v_n^\sigma$ be ray generators for $\sigma$.
Denote by $L_\sigma$ the lattice
in $\R^n$ that is generated by these vectors.
Then $\sigma$ is smooth with respect to
$L_\sigma$. Notice that our point of view here is slightly different from the classical one, where the
lattice is given from the beginning, and
the cone is required to be smooth with respect to that lattice. As we shall see, our approach is better suited to being generalized to the nonrational case. 

The intersection $$S_{\sigma}=\sigma^{\vee}\cap L_\sigma^{\vee}$$ is a commutative semigroup with $0$ that is generated by the basis $\alpha_1^\sigma,\ldots,\alpha_n^\sigma$ that is dual to the basis $v_1^\sigma,\ldots,v_n^\sigma$. 
The corresponding semigroup algebra $\C[S_\sigma]$ is  the set of finite formal sums $\sum_{\alpha\in S_\sigma}c_\alpha\chi^{\alpha}$ with $c_\alpha\in \C$ and multiplication rule given by $\chi^\alpha\cdot\chi^{\alpha'}=\chi^{\alpha+\alpha'}$,
$\alpha, \alpha'\in S_\sigma$. Note that any $\alpha\in S_\sigma$ is of the form $h_1\alpha_1^\sigma+\cdots+h_n\alpha_n^\sigma$, with $h_1,\ldots,h_n$ nonnegative integers. 
Therefore, if we denote $\xi_i=\chi^{\alpha_i}$, $i=1,\ldots,n$, we obtain $\C[S_\sigma]=\C[\xi_1,\ldots,\xi_n]$, the ring of polynomials in $\xi_1,\ldots,\xi_n$ with complex coefficients. The associated affine toric variety is the set of all maximal ideals of $\C[S_\sigma]$, that we denote for brevity $\text{Spec}\,\C[S_\sigma]$ instead of $\text{Specm}\,\C[S_\sigma]$. By identifying the maximal ideal $\langle \xi_1-z_1,\ldots,\xi_n-z_n\rangle$ with the point $(z_1,\ldots,z_n)\in\C^n$, we obtain that the affine toric variety is $\C^n$ with coordinate ring $\C[S_\sigma]=\C[\xi_1,\ldots,\xi_n]$. 

Take now any $r$--dimensional simplicial cone $\delta$ in $\R^n$. Choose a basis $v_1^\delta,\ldots,v_n^\delta$
of $\R^n$ such that  $v_1^\delta,\ldots,v_r^\delta$ are 
ray generators of $\delta$ and
consider the lattice $L_\delta$ that is generated by $v_1^\delta,\ldots,v_n^\delta$.
The cone $\delta$ is smooth with respect to
$L_\delta$.
Consider now the semigroup
$$S_{\delta}=\delta^{\vee}\cap L_\delta^{\vee}.$$
Then, taking the basis $\{\alpha_1^\delta,\ldots,\alpha_n^\delta\}$ dual to the basis $\{v_1^\delta,\ldots,v_n^\delta\}$, 
we have
$$\begin{array}{ccl}\C[S_{\delta}]&=&\C[\chi^{\alpha_1^\delta},\ldots,\chi^{\alpha_r^\delta},\chi^{\alpha_{r+1}^\delta}, \chi^{-\alpha_{r+1}^\delta},\ldots,\chi^{\alpha_{n}^\delta},\chi^{-\alpha_{n}^\delta}]\\&=&\C[\xi_1,\ldots,\xi_r,\xi_{r+1},\xi_{r+1}^{-1},\ldots,\xi_n,\xi_n^{-1}],\end{array}$$ which is a subalgebra of the algebra of all Laurent polynomials in $\xi_1,\ldots,\xi_n$. 
Consider the ideal $$J=\langle z_{r+1}z_{r+2}-1,\ldots,z_{2n-r-1}z_{2n-r}-1 \rangle$$ in $\C[z_1,\ldots,z_{2n-r}]$. Then $\C[S_{\delta}]$ is isomorphic to $\C[z_1,\ldots,z_{2n-r}]/J$  and the affine toric variety 
$\text{Spec}\,\C[z_1,\ldots,z_{2n-r}]/J$ corresponding to $\delta$ can be identified with the set of zeros $$V(J)=\{(z_1,\ldots,z_{2n-r})\in\C^{2n-r}\mid z_{r+1}z_{r+2}=\cdots=z_{2n-r-1}z_{2n-r}=1\}.$$ In turn, this can be identified with $\C^r\times(\C^*)^{n-r}$ by the projection from $\C^r\times\C^{2n-2r}$ onto $\C^r\times(\C^*)^{n-r}$ that is the identity on $\C^r$ and on $\C^{2n-2r}$ is given by $$\begin{array}{ccc}\C^{2n-2r}&\longrightarrow&(\C^*)^{n-r}\\
  (z_{r+1},z_{r+2},z_{r+3}\ldots,z_{2n-r-1},z_{2n-r})&\longmapsto&
  (z_{r+1},z_{r+3},\ldots,z_{2n-r-1})\end{array}.$$ Notice that, in particular, to the $0$--dimensional cone $\{0\}$ there corresponds $\C\left[S_{\{0\}}\right]=\C\left[L_{\{0\}}^\vee\right]$, generated 
  by $\left\{\chi^{\alpha_1^{\{0\}}},\ldots,\chi^{\alpha_n^{\{0\}}},\chi^{-\alpha_1^{\{0\}}},\ldots,\chi^{-\alpha_n^{\{0\}}}\right\}$, where $\left\{\alpha_1^{\{0\}},\dots,
  \alpha_n^{\{0\}}\right\}$ is now a generic basis of  $(\R^n)^*$.
  This, with the above identifications, yields the torus $(\C^*)^n$. 

  We illustrate these constructions with two examples.
\begin{example}\label{quasispheremodel}
\rm{
Consider the maximal cone $\sigma=\R_{\geq 0}\subset\R$. Choose the generator $v_1=a$, with $a$ any positive real number, and take the corresponding lattice $L_\sigma=a\Z$. The
dual basis element $\alpha_1^\sigma=\frac{1}{a}e_1^*$
generates the dual cone $\sigma^{\vee}$ and the dual lattice $L_\sigma^{\vee}$ in $\R^*$.
Consider the semigroup $$S_{\sigma}=\sigma^{\vee}\cap L_\sigma^{\vee}=\left\{\frac{h}{a}e_1^*\mid h\in\Z,h\geq0\right\}.$$ The algebra $\C[S_\sigma]$ is given by $\C[\xi_1]$ and the corresponding affine toric variety is $\C$.
Take now the face $\{0\}$.
Then $S_{\{0\}}=L_\sigma^{\vee}$,
$\C[S_{\{0\}}]=\C[\xi_1, \xi_1^{-1}]$
and the corresponding affine toric variety is $\C^*$.}
\end{example}
\begin{example}\label{cono proiettivo} 
\rm{
Take the maximal cone $\tau$ in $\R^2$ generated by the vectors $v_1=(1,0), v_2=(-1,-a)$, with $a$ a positive
real number (see Figure~\ref{trianglecone}). 
\begin{figure}[h]
\begin{center}
\includegraphics{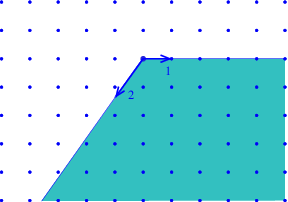}
\caption{The maximal cone $\tau$.}
\label{trianglecone}
\end{center}
\end{figure}
Consider the generated lattice $L_\tau=\Z\times a\Z$.
The dual cone $\tau^{\vee}$ and the dual lattice $L_\tau^{\vee}$ in $(\R^2)^*$ are generated by the vectors of the dual basis $\alpha_1^\tau=e_1^*-\frac{1}{a}e_2^*, \alpha_2^\tau=-\frac{1}{a}e_2^*$. Consider the semigroup $$S_{\tau}=\tau^{\vee}\cap L_\tau^{\vee}=\left\{h\left(e_1^*-\frac{1}{a}e_2^*\right)-\frac{k}{a}e_2^*\mid h,k\in\Z,h,k\geq0\right\}.$$ The algebra $\C[S_\tau]$ is given by $\C[\xi_1,\xi_2]$ and the corresponding affine toric variety is $\C^2$.
Now take the face $\delta$ generated by $v_1$. Together with $v_2$ they form a basis
of $\R^2$ such that  $L_\delta=L_\tau$. The corresponding semigroup is $$\begin{array}{ll}S_{\delta}\!\!\!&=\delta^{\vee}\cap L_\tau^{\vee}\\&=\{h\left(e_1^*-\frac{1}{a}e_2^*\right)-\frac{k}{a}e_2^*\mid h,k\in\Z,\, \left(h\left(e_1^*-\frac{1}{a}e_2^*\right)-\frac{k}{a}e_2^*\right)(e_1) \geq 0\}\\ &=\{h\left(e_1^*-\frac{1}{a}e_2^*\right)-\frac{k}{a}e_2^*\mid h,k\in\Z,\,h\geq0\}.\end{array}$$ Thus the semigroup algebra $\C[S_{\delta}]$ is given by $\C[\xi_1,\xi_2,\xi_2^{-1}]$, which is the coordinate ring of the affine toric variety $\C\times\C^*$.}
\end{example}

\subsection{Finite quotient singularities}\label{finite}
Let us go back to a general maximal simplicial cone $\sigma$ with ray generators given by $v_1^\sigma,\ldots,v_n^\sigma$.
We now consider new
ray generators $\frac{v^\sigma_1}{h_1},\ldots,\frac{v^\sigma_1}{h_n}$ for $\sigma$,
with $h_1,\ldots,h_n$ positive integers; they span a new lattice $\widetilde{L}_\sigma$ of which $L_{\sigma}$ is a sublattice having the same rank. The vectors $v_1^\sigma,\dots,v_n^\sigma$ are no longer primitive in $\widetilde{L}_\sigma$. 
Let us recall the intrinsic procedure described in \cite[p.~34]{fulton}, which yields the
affine toric variety (and its coordinate ring) corresponding
to the cone $\sigma$ with respect to the new lattice $\widetilde{L}_\sigma$. We do this
because, as we will see, this procedure extends to the nonrational case naturally, by replacing the larger lattice $\widetilde{L}_\sigma$ with a quasilattice $Q$.

Consider the canonical duality pairing
$$L_\sigma ^{\vee}/\widetilde{L}_\sigma^{\vee}\times \widetilde{L}_\sigma/L_\sigma\rightarrow \Q/\Z \hookrightarrow \C^*$$ that takes
$([\alpha],[v])$ to $e^{-2\pi i\alpha(v)}$,
 for $\alpha \in L_\sigma ^{\vee}, v\in \widetilde{L}_\sigma$.
This induces an action of the finite group  $\Gamma_\sigma=\widetilde{L}_\sigma/L_\sigma$ on $\C[S_{\sigma}]$:
$$[v]\cdot\chi^\alpha=e^{-2\pi i\alpha(v)}\chi^\alpha,\quad \alpha\in S_{\sigma}, v\in \widetilde{L}_\sigma.$$  Then, given the semigroup
$$\widetilde{S}_{\sigma}=\sigma^{\vee}\cap \widetilde{L}_\sigma^{\vee},$$
we have that
$$\C[\widetilde{S}_\sigma]=\C[S_\sigma]^{\Gamma_\sigma},$$
the latter being the ring of invariants.
Moreover, there is an induced action of $\Gamma_\sigma$
on $\text{Spec}\,\C[S_{\sigma}]$, and consequently on $\C^n$, defined as follows: on the generators of the maximal ideals we have, for $v\in\widetilde{L}_\sigma$:
$$\begin{array}{r}[v]\cdot(\chi^{\alpha_j^\sigma}-z_j\chi^0)=e^{-2\pi i\alpha_j^\sigma(v)}\chi^{\alpha_j}-z_j\chi^0=\\=e^{-2\pi i\alpha_j^\sigma(v)}(\chi^{\alpha_j^\sigma}-e^{2\pi i\alpha_j^\sigma(v)}z_j\chi^0).\end{array}$$
Therefore, the action of $\Gamma_\sigma$ on $\C^n$ is given by $$[v]\cdot(z_1,\ldots,z_n)=(e^{2\pi i\alpha_1^\sigma(v)}z_1,\ldots,e^{2\pi i\alpha_n^\sigma(v)}z_n),\quad v\in \widetilde{L}_\sigma.$$ Finally, the affine toric variety corresponding to the cone $\sigma$ relative to the lattice $\widetilde{L}_\sigma$ is given by the quotient $\C^n/\Gamma_\sigma$.
Something analogous happens for nonmaximal simplicial cones $\delta$. We will treat this case below, directly in the nonrational setting.

\subsection{The nonrational case}
We are now ready to replace, as announced, the lattice $\widetilde{L}_\sigma$ with a quasilattice
$Q$ and to introduce the notion
of \textit{affine toric quasifold}. We do this as follows. Consider first a triple 
$(\sigma, Q, \{v_1^\sigma, \ldots, v_n^\sigma\})$,
with $\sigma$ maximal. Then $L_\sigma\subset Q$
and $Q^{\vee}\subset L_\sigma^{\vee}$.
Formally, everything goes as above.
Namely, there is a canonical duality pairing  
$$L_\sigma ^{\vee}/Q^{\vee}\times Q/L_\sigma\rightarrow \R/\Z \hookrightarrow \C^*$$ which again takes
$([\alpha],[v])$ to $e^{-2\pi i\alpha(v)}$,
for $\alpha \in L_\sigma^{\vee}, v\in Q$. The now countable group  $\Gamma_\sigma=Q/L_\sigma$ acts on $\C[S_{\sigma}]$ as above:
$$[v]\cdot\chi^\alpha=e^{-2\pi i\alpha(v)}\chi^\alpha, \quad \alpha\in S_{\sigma}, v\in Q$$
(notice that if $\alpha\in Q^{\vee}$, then $\chi^{\alpha}$
is invariant). We thus have an induced
action of $\Gamma_\sigma$ on $\C^n$, which is again given by $$[v]\cdot(z_1,\ldots,z_n)=(e^{2\pi i\alpha_1^\sigma(v)}z_1,\ldots,e^{2\pi i\alpha_n^\sigma(v)}z_n),\quad v\in Q.$$
In analogy with the rational case, we are now
ready to define affine toric quasifolds,
for maximal simplicial cones.

\begin{dfn}\rm{Given a triple
$$\{\sigma,Q,\{v_1^\sigma,\ldots,v_n^\sigma\}\},$$ with
$\sigma$ a maximal simplicial cone, we define 
the corresponding \textit{affine toric quasifold} to be the quotient
$$X_{\sigma}=\C^n/\Gamma_\sigma.$$ We will say that $n$ is the \textit{dimension} of $X_{\sigma}$.}
\end{dfn}
Let us now consider the case of a general
simplicial cone $\delta$ of dimension $r$ and take a triple 
$(\delta, Q, \{v_1^\delta, \ldots, v_n^\delta)\}$. For simplicity, we assume $\delta$ to be generated by $v_1^\delta,\ldots,v_r^\delta$. Then the action of $\Gamma_\delta=Q/L_\delta$ on $\C[S_{\delta}]$ yields the following action
on the isomorphic ring $\C[z_1,\ldots,z_{2n-r}]/J$: let $[v]\in\Gamma_\delta$, then $$[v]\cdot z_h=\left\{\begin{array}{ll}e^{-2\pi i\alpha_h^\delta(v)}z_h,&  h=1,\ldots,r\\ 
e^{-2\pi i\alpha_h^\delta(v)}z_{r+2k+1},&  k=0,\ldots,n-r-1\\ 
e^{2\pi i\alpha_h^\delta(v)}z_{r+2k} & k=1,\ldots,n-r.\end{array}\right.$$
Keeping track of all the identifications, the action of $\Gamma_\delta$ is just the restriction to $\C^r\times(\C^*)^{n-r}$ of the action on $\C^n$ given by
$$[v]\cdot(z_1,\ldots,z_n)=(e^{2\pi i\alpha_1^\delta(v)}z_1,\ldots,e^{2\pi i\alpha_n^\delta(v)}z_n), \quad v\in Q.$$
\begin{dfn} 
\rm{Given a triple
$\{\delta,Q,\{v_1^\delta,\ldots,v_n^\delta\}\}$, with
$\delta$ an $r$--dimensional simplicial cone, we define the corresponding \textit{affine toric quasifold} to be the quotient
$$X_{\delta}=(\C^r\times(\C^*)^{n-r})/\Gamma_\delta.$$
We will say that $n$ is the \textit{dimension} of $X_{\delta}$.}
\end{dfn}
\begin{rem}\label{facce}
\rm{
When $\rho$ is a face of $\delta$, 
we will denote
the affine quasifold corresponding to the induced triple $\{\rho,Q,\{v_1^\delta,\ldots,v_n^\delta\}\}$ by $X_\rho^\delta$, so as not to lose track of the ambient cone $\delta$.}
\end{rem}
For some applications, it will be useful to
have an alternate description of the
group $\Gamma_\delta$.
\begin{lem}\label{gruppino}
Consider a triple
$(\delta,Q,\{v_1^\delta,\ldots,v_n^\delta\})$ and the simplex
$$\Delta_\delta=\left\{ \sum_{i=1}^n s_i v_i^\delta \mid 0\leq s_i<1\right\}.$$
Then we have
$$\Gamma_\delta=p( Q\cap \Delta_\delta),
$$
where $p$ is the projection
$$
\begin{array}{ccl}
&Q\rightarrow &\Gamma_\delta\\
& v \mapsto &[v].
\end{array}
$$
\end{lem}
\noindent\begin{preuve}
Take any $v\in Q$ and write it as $v=\sum_{i=1}^n a_i v_i^\delta$ (note that $a_i=\alpha_i^{\sigma}(v)$). 
Now decompose each $a_i$ as the sum of its integer part
and its fractional part:
$a_i=\lfloor a_i \rfloor +\{a_i\}$. Then
$p(v)=p(u)$, with $u=\sum_{i=1}^n \{ a_i \} v_i^\delta\in Q\cap \Delta_\delta$.
\end{preuve}

We illustrate the notion of affine toric quasifold with two important examples.
\begin{example}
\label{quasispheremodeltwo}
\rm{In Example~\ref{quasispheremodel}, take the quasilattice $Q=\Z+a\Z\supset a\Z=L_\sigma$. The affine toric quasifold associated with the triple
$(\sigma,Q,\{a\})$ is given by $X_\sigma=\C/\Gamma_\sigma$,
where $$\Gamma_\sigma=
\left\{ \left[\frac{h}{a}v_1\right]\mid h\in \Z\right\}$$ acts on $\C$ by 
rotations of angle $\frac{1}{a}$:
$$\left[\frac{h}{a}v_1\right]\cdot z=e^{2\pi i \frac{h}{a}}z.$$
Similarly, $X_{\{0\}}^\sigma=\C^*/\Gamma_\sigma$.
Notice that, when $a=1$, we get $X_\sigma=\C$, whilst if $a$ equals
a positive integer $n$ we get $X_\sigma=\C/\Z_n$, which is isomorphic to $\C$.}
\end{example}
Notice that, strictly speaking, in accordance with Lemma~\ref{gruppino}, the fraction $\frac 1a$ in Example~\ref{quasispheremodeltwo} above can be
replaced by its fractional part $\left\{\frac 1a \right\}$. We are, in fact, rotating by an angle $\left\{\frac 1a \right\}$. However, we will avoid doing so here, and in the following examples, as the notation is cumbersome.
\begin{example}\label{cono proiettivo 2}
\rm{
Consider the cone $\tau$ in Example~\ref{cono proiettivo}, and take 
the quasilattice
$$Q=\Z\times(\Z+a\Z)\supset \Z\times a\Z=L_\tau.$$ 
The affine toric quasifold corresponding to the triple $(\tau,Q,\{v_1,v_2\})$ is  $X_\tau=\C^2/\Gamma_\tau$ and the one corresponding to the face $\delta$ is $X_\delta^\tau=(\C\times\C^*)/\Gamma_\tau$, where
$$\Gamma_\tau=\left\{\ \left[\frac{h}{a}(v_1+v_2)\right] \mid h\in \Z\right\}$$
acts on $\mathbb{C}^2$ as follows
$$\left[\frac{h}{a}(v_1+v_2)\right]\cdot(z_1,z_2)=\left(e^{2\pi i \frac{h}{a}}z_1,e^{2\pi i \frac{h}{a}}z_2\right).$$
When $a=1$ the group $\Gamma_\tau$ is trivial and $X_\tau=\C^n$.
On the other hand, when $a$ is a positive integer $n$, then $\Gamma_\tau= \Z_n$. In this case,
$Q^\vee=(\Z^2)^\vee\simeq\Z^2\subset L^{\vee}=\Z\times\frac{1}{n}\Z$ and $X_\tau$ is an orbifold with singularity of order $n$. It is the affine variety corresponding to the singular cone in ordinary weighted projective space
$\C\P^2_{(1,1,n)}$.}
\end{example}
\begin{rem}
\rm{We have seen in Subsection~\ref{finite} that a smooth maximal simplicial cone $\sigma$, with ray generators given by $v_1^\sigma,\ldots,v_n^\sigma$, gives rise to the coordinate ring 
$\C[S_\sigma].$ When we consider new ray generators
$\frac{v^\sigma_1}{h_1},\ldots,\frac{v^\sigma_1}{h_n}$
with $h_1,\ldots,h_n$ positive integers, and therefore a new lattice
$\widetilde{L}_\sigma$ containing $L_{\sigma}$, then its coordinate ring is the ring of invariants
$\C[\widetilde{S}_\sigma]=\C[S_\sigma]^{\Gamma_\sigma},$ where $\Gamma_\sigma=\widetilde{L}_\sigma/L_\sigma$.

When we replace the lattice
$\widetilde{L}_\sigma$ with a quasilattice $Q$, the change is dramatic. As we have seen in Remark~\ref{duali}, the quasilattice $\text{Hom}(Q,\Z)$ can no longer be identified with $Q^\vee$. The latter, as a matter of fact, is a sublattice of $L^\vee_\sigma$. Even if we had the notion of a dual quasilattice $Q^*$ sitting in $(\R^n)^*$, the semigroup $\sigma^\vee\cap Q^*$, as we have seen, would be not finitely generated. Therefore, while the classical construction of an affine variety with finite quotient singularities leads very naturally to that of an affine toric quasifold, the notion of ring of functions cannot be extended in this way.
}
\end{rem}
\subsection{Quasitori as affine toric quasifolds}
The affine toric quasifold corresponding to the cone $\{0\}\subset \delta$ is a very special one: $X_{\{0\}}^\delta=(\C^*)^{n}/\Gamma_\delta$. As we shall see, this is what is called a \textit{quasitorus}.
\begin{dfn}{\rm Let $Q$ be a quasilattice in $\R^n$, the group $\R^n/Q$ is called a 
\textit{real quasitorus} and the group $\C^n/Q$ is called an \textit{algebraic quasitorus}. 
}    
\end{dfn}
Notice that, in the smooth case, when $Q$ coincides with $L_\delta$, these groups are, respectively, a compact real torus and an algebraic torus. Moreover, in this case the mapping
$$\begin{array}{ccccc}\exp_\delta&\colon&\R^n&\longrightarrow&(S^1)^n\\&&u&\longmapsto&(e^{2\pi i\alpha_1^{\delta}(u)},\ldots,e^{2\pi i \alpha_n^\delta(u)})
\end{array}.
$$
induces an isomorphism from $\R^n/L_\delta$ to $(S^1)^n$, while its complex version induces
an isomorphism from $\C^n/L_\delta$ to $(\C^*)^n$. 

In the general case of $Q\supset L_\delta$, notice that $\exp_\delta$ induces an isomorphism from $\Gamma_\delta=Q/L_\delta$ to the subgroup of $(S^1)^n$ given by 
$$\{
(e^{2\pi i\alpha_1^\delta(v)},\ldots,e^{2\pi i\alpha_n^\delta(v)})\in (S^1)^n\;|\; v\in Q\}$$ which we will continue to call $\Gamma_\delta$.
Finally, $\exp_\delta$ and its complex version, induce group isomorphisms of $\R^n/Q$ and $\C^n/Q$  onto $(S^1)^n/\Gamma_\delta$ and $(\C^*)^n/\Gamma_\delta$ respectively. 
Thus 
$X_{\{0\}}^\delta$ is a quasitorus.
\section{Morphisms of affine toric quasifolds}\label{affinemorphism}
We are now ready to introduce, in this setting, the notion of affine toric morphism.
\subsection{Definition and properties}
Consider triples $(\delta_1, Q_1, \{v^{\delta_1}_1,\ldots,v^{\delta_1}_n\})$ and $(\delta_2, Q_2, \{v^{\delta_2}_1,\ldots,v^{\delta_2}_m\})$ and a cone morphism
$\delta_1\rightarrow \delta_2$ defined by a linear map $f\colon\R^n\rightarrow\R^m$.
Take the basis $\alpha_1^{\delta_1},\ldots,\alpha_n^{\delta_1}$ dual
to the basis $v^{\delta_1}_1,\ldots, v^{\delta_1}_n$ and the basis $\alpha_1^{\delta_2},\ldots,\alpha_m^{\delta_2}$ dual
to $v_1^{\delta_2},\ldots, v_m^{\delta_2}$.
Then there is an induced mapping between the
corresponding affine toric quasifolds, 
which is defined as follows:
\begin{equation}\label{morfismo}
\begin{array}{cclc}
\Phi_{\delta_2\delta_1}\colon & X_{\delta_1}&\longrightarrow & X_{\delta_2} \\
& [z_1:\cdots:z_n]&\longmapsto & \left[\prod\limits_{h=1}^nz_h^{\alpha_1^{\delta_2}(f(v_h^{\delta_1}))}:\cdots:\prod\limits_{h=1}^nz_h^{\alpha_m^{\delta_2}(f(v_h^{\delta_1}))}\right].
\end{array}
\end{equation}
\begin{thm}\label{welldefined}
The mapping $\Phi_{\delta_2\delta_1}$ is well defined.
\end{thm}
\noindent\begin{preuve} 
Take a point
$[z_1:\cdots:z_n]\in X_{\delta_1}$.
Notice that $z_h$ can be $0$ if and only if $h\in I_{\delta_1}$. However, the corresponding exponents are non negative.
In fact, $f(\delta_1)\subset\delta_2$ implies $f^*(\delta_2^\vee)\subset\delta_1^\vee$. Therefore, since $\alpha_j^{\delta_2}\in\delta_2^\vee$,  $j=1,\ldots,m$, we have that $\alpha_j^{\delta_2}(f(v_h^{\delta_1}))\geq0$ for all $h\in I_{\delta_1}$.
We adopt the convention that $0$ to the power $0$ gives $1$. Now, write each $z_h\neq0$
as $z_h=e^{2\pi i x_h}$, for
a choice of $x_h\in\R$. Notice, however, that
we also have $z_h=e^{2\pi i (x_h+k_h)}$ for every $k_h\in \Z$. 
Taking $u=\sum_{h=1}^n k_h v_h^{\delta_1}$, 
for all $k_h\in\Z$, $h=1,\ldots,n$, we can express this as
$z_h=e^{2\pi i (x_h+\alpha_h^{\delta_1}(u))}$, for every $u \in L_{\delta_1}$.
This entails that  each
monomial appearing on the right hand side of (\ref{morfismo})
$$\prod_{h=1}^n z_h^{\alpha_j^{\delta_2}(f(v_h^{\delta_1}))}, \quad j=1,\ldots ,m$$ is unique
up to 
$$\prod_{h=1}^n e^{2\pi i \alpha_j^{\delta_2}(f(v_h^{\delta_1}))\alpha_h^{\delta_1}(u)}
= e^{2\pi i \left(\sum_{h=1}^n \alpha_j^{\delta_2}(f(v_h^{\delta_1}))\alpha_h^{\delta_1}(u)\right)}=e^{2\pi i \alpha_j^{\delta_2} \left(f\left(u\right)\right)},$$
$u\in L_{\delta_1}$. However, 
in $X_{\delta_2}$ this is not an issue
at all, since $f(u)\in Q_2$
and this amounts to the action of an element
of $\Gamma_{\delta_2}$.

Let us now show that $\Phi_{\delta_2\delta_1}$
does not depend on
the particular representative
of each point in $X_{\delta_1}$, namely that
$$\Phi_{\delta_2\delta_1}\left[e^{2\pi i\alpha_1^{\delta_1}(v)}z_1:\cdots:e^{2\pi i\alpha_n^{\delta_1}(v)}z_n\right]=\Phi_{\delta_2\delta_1}[z_1:\cdots:z_n]$$
for every vector $v\in Q_1$ and for every point $[z_1:\cdots:z_n]\in X_{\delta_1}$. However, we have, for $j=1,\ldots,m$, (again up to
$e^{2\pi i \alpha_j^{\delta_2} \left(f\left(u\right)\right)}$, $u_{\delta_1}\in L_{\delta_1}$)

$$\begin{array}{cll}\prod\limits_{h=1}^n\left(e^{2\pi i\alpha_h^{\delta_1}(v)}z_h\right)^{\alpha^{\delta_2}_{j}(f(v_h^{\delta_1}))}
&=&\prod\limits_{h=1}^n e^{2\pi i \alpha_{j}^{\delta_2}(f(v_h^{\delta_1}))\alpha_h^{\delta_1}\left(v\right)}\prod\limits_{h=1}^n z_h^{\alpha_{j}^{\delta_2}(f(v_h^{\delta_1}))}\\
&=&e^{2\pi i\sum_{h=1}^n 
\alpha_{j}^{\delta_2}(f(v_h^{\delta_1}))\alpha_h^{\delta_1}(v)}\prod\limits_{h=1}^n z_h^{\alpha_{j}^{\delta_2}(f(v_h^{\delta_1}))}\\
&=&e^{2\pi i\alpha^{\delta_2}_{j}(f(v))}\prod\limits_{h=1}^n {z_h}^{\alpha^{\delta_2}_{j}(f(v_h^{\delta_1}))}.
\end{array}$$
and, again, since $f(v)\in Q_2$, this amounts to acting
via an element of $\Gamma_{\delta_2}$ and
thus has no effect on $X_{\delta_2}$.
\end{preuve}

\begin{dfn}[Affine toric morphism]\rm{A \textit{morphism} between two affine quasifolds $X_{\delta_1}$ and $X_{\delta_2}$  is any mapping that is induced by a cone morphism.} 
\end{dfn}
\begin{dfn}[Generalized Laurent monomials]
\rm{We call the monomials
$$\prod_{h=1}^n
z_h^{\alpha_j^{\delta_2}(f(v_h^{\delta_1}))}, \quad j=1,\ldots ,m$$
appearing in the definition of $\Phi_{\delta_2\delta_1}$
\textit{generalized
Laurent monomials}.}
\end{dfn}
\begin{rem}
\rm{As we have shown in the proof of
Theorem~\ref{welldefined} above, these monomials  are only well--defined up to the elements
$e^{2\pi i \alpha_j^{\delta_2}(f(u))}$, $u\in L_{\delta_1}$. Unlike what happens in the rational case, as standalones, they cannot be thought of as elements of
the algebra of Laurent polynomials. They only make sense when grouped together and projected onto $X_{\delta_2}$.

One possible formal interpretation is the following. We assume for simplicity that $\delta_1$ and $\delta_2$ are maximal cones. Recall then that 
\begin{itemize}
\item $\C[S_{\delta_1}]=\C[\xi_1^{\delta_1},\ldots,\xi^{\delta_1}_n]$, where $\xi_h^{\delta_1}=\chi^{\alpha_h^{\delta_1}}$, $h=1,\ldots,n$
\item $\C[S_{\delta_2}]=\C[\xi_1^{\delta_2},\ldots,\xi^{\delta_2}_m]$, where $\xi_j^{\delta_2}=\chi^{\alpha_j^{\delta_2}}$, $j=1,\ldots,m$.
\end{itemize}
Define
the following formal mapping from $\C[S_{\delta_2}]$ to the set of 
formal polynomials containing $\C[S_{\delta_1}]$:
$$\Phi_{\delta_2\delta_1}^*\left(\xi_j^{\delta_2}\right)=\prod_{h=1}^n\left(\xi_h^{\delta_1}\right)^{\alpha_j^{\delta_2}(f(v_h^{\delta_1}))},\quad j=1,\ldots ,m.$$
If we interpret these polynomials as formal functions on $\C^n$ and $\C^m$ respectively, this gives rise to a formal mapping $\widetilde{\Phi}_{\delta_2\delta_1}\colon\C^n\rightarrow\C^m$ by requiring, like in the rational case, that $\Phi_{\delta_2\delta_1}^*$ is the pullback of $\widetilde{\Phi}_{\delta_2\delta_1}$,
namely that 
$$
\xi^{\delta_2}_j\left(\widetilde{\Phi}_{\delta_2\delta_1}(z_1,\ldots, z_n)\right)=\Phi_{\delta_2\delta_1}^*\left(\xi^{\delta_2}_j\right)(z_1,\ldots, z_n),\quad\quad j=1,\ldots,m.
$$
We then necessarily have
$$\widetilde{\Phi}_{\delta_2\delta_1}(z_1,\ldots,z_n)=\left(\prod_{h=1}^n z_h^{\alpha_1^{\delta_2}(f(v_h^{\delta_1}))},\ldots,\prod_{h=1}^n z_h^{\alpha_m^{\delta_2}(f(v_h^{\delta_1}))}\right),$$
which induces our mapping
$$
\Phi_{\delta_2\delta_1}\colon X_{\delta_1} \longrightarrow X_{\delta_2}.
$$}
\end{rem}
We now prove that the map sending a cone morphism to the induced affine morphism has the functorial property, namely it respects composition and identity.
\begin{thm}\label{functoriality}
Consider triples
$(\delta_1, Q_1, \{v^{\delta_1}_1,\ldots v^{\delta_1}_n\})$, $(\delta_2, Q_2, \{v^{\delta_2}_1,\ldots v^{\delta_2}_l\})$, and $(\delta_3, Q_3, \{v^{\delta_3}_1,\ldots v^{\delta_3}_m\})$. Let $\delta_1\rightarrow\delta_2$ and 
$\delta_2\rightarrow\delta_3$
be cone morphisms defined, respectively, by linear mappings $f\colon \R^n\rightarrow \R^l$
and $g\colon \R^l\rightarrow \R^m$. Let $\Phi_{\delta_2\delta_1}\colon X_{\delta_1}\rightarrow X_{\delta_2}$ and $\Psi_{\delta_3\delta_2}\colon X_{\delta_2}\rightarrow X_{\delta_3}$ be the corresponding affine toric morphism. Then
$\Psi_{\delta_3\delta_2}\circ \Phi_{\delta_2\delta_1}\colon X_{\delta_1}\rightarrow X_{\delta_3}$
is the morphism induced by $g\circ f$.
Moreover, the identity cone morphism 
induces the identity affine quasifold morphism.
\end{thm}
\noindent\begin{preuve}
$$\begin{array}{l}(\Psi_{\delta_3\delta_2}\circ \Phi_{\delta_2\delta_1})
[z_1:\cdots:z_n]\\=\Psi_{\delta_3\delta_2}\left[\prod\limits_{h=1}^n z_h^{\alpha_1^{\delta_2}(f(v_h^{\delta_1}))}:\cdots:\prod\limits_{h=1}^n z_h^{\alpha_l^{\delta_2}(f(v_h^{\delta_1}))}\right]\\=
\left[\prod\limits_{j=1}^l\left(\prod\limits_{h=1}^n z_h^{\alpha_j^{\delta_2}(f(v_h^{\delta_1}))}\right)^{\alpha_1^{\delta_3}(g(v_j^{\delta_2}))}:\cdots:\prod\limits_{j=1}^l\left(\prod\limits_{h=1}^n z_h^{\alpha_j^{\delta_2}(f(v_h^{\delta_1}))}\right)^{\alpha_m^{\delta_3}(g(v_j^{\delta_2}))}\right]\\
=\left[\prod\limits_{h=1}^n z_h^{\sum_{j=1}^l\alpha_1^{\delta_3}(g(v_j^{\delta_2}))\alpha_j^{\delta_2}(f(v_h^{\delta_1}))}:\cdots:\prod\limits_{h=1}^n z_h^{\sum_{j=1}^l\alpha_m^{\delta_3}(g(v_j^{\delta_2}))\alpha_j^{\delta_2}(f(v_h^{\delta_1}))}\right]\\
=\left[\prod\limits_{h=1}^n z_h^{\alpha_1^{\delta_3}((g\circ f)(v_h^{\delta_1}))}:\cdots:\prod\limits_{h=1}^n z_h^{\alpha_m^{\delta_3}((g\circ f)(v_h^{\delta_1}))}\right],
\end{array}
$$
In the same way, it is easy to check that the identity cone morphism induces the identity affine quasifold isomorphism.
\end{preuve}
\begin{dfn}[Affine toric isomorphism] \rm{An \textit{isomorphism} between two affine quasifolds $X_{\delta_1}$ and $X_{\delta_2}$  is any mapping that is induced by a cone isomorphism. Whenever this happens we will say that two affine toric quasifolds are \textit{isomorphic}.}
\end{dfn}
It is easy to check, in fact, that 
such a map is bijective.
\subsection{Cone inclusions}
Cone inclusions provide some interesting examples of morphisms. We begin with the important special case of face inclusions.
\begin{lem}\label{faceinclusion}
Consider a triple
$({\delta}, Q, \{v^{\delta}_1,\ldots, v^{\delta}_n\})$ and
a face $\rho\subset\delta$. There is a natural inclusion of
$X_{\rho}^\delta$ in $X_{\delta}$.
\end{lem}

\noindent\begin{preuve}
The identity of $\R^n$ defines a cone morphism $\rho\rightarrow \delta$ that induces the required inclusion. 
\end{preuve}

Note that both quasifolds have dimension $n$.

Consider now two triples $({\delta_1}, Q_1, \{v^{\delta_1}_1,\ldots, v^{\delta_1}_n\})$,    $({\delta_2}, Q_2, \{v^{\delta_2}_1,\ldots, v^{\delta_2}_m\})$
and a cone morphism $\delta_1\rightarrow \delta_2$
defined by a linear map $f\colon \R^n\rightarrow \R^m$.
Assume that we have faces
$\rho_1\subset \delta_1$, $\rho_2\subset \delta_2$ such that $f(\rho_1)\subset\rho_2$.
The above morphism restricts to a morphism
$\rho_1\rightarrow \rho_2$.
\begin{prop}\label{restriction}
$\Phi_{\rho_2\rho_1}$ is the restriction of 
$\Phi_{\delta_2\delta_1}$ to $X_{\rho_1}^{\delta_1}$.
\end{prop}
\noindent\begin{preuve}
By Theorem~\ref{functoriality},
the commutative diagram
$$	
\xymatrix{
\rho_1\ar[r]\ar@{_(->}[d]&\rho_2
\ar@{_(->}[d]\\
\delta_1\ar[r]&\delta_2.
}
$$
induces a commutative diagram of
the corresponding affine quasifold morphisms
$$	
\xymatrix{
X_{\rho_1}^{\delta_1}\ar[r]^{\Phi_{\rho_2\rho_1}}\ar@{_(->}[d]&X_{\rho_2}^{\delta_2}
\ar@{_(->}[d]\\
X_{\delta_1}\ar[r]^{\Phi_{\delta_2\delta_1}}&X_{\delta_2},
}
$$
where the vertical arrows are natural inclusions by Lemma~\ref{faceinclusion}.

\end{preuve}

This result will be of relevance in the case where the cones are part of a fan and the morphisms descend from a fan morphism.

The situation for general cone inclusions is quite different, as we can see from the important example below.
\begin{example}\label{inclusion}
\rm{Consider the cone
$\tau$ of Examples~\ref{cono proiettivo}, \ref{cono proiettivo 2} and take the cone
$\eta$ that is generated by the vectors $v_2=(-1,-a)$ and
$v_4=(0,-1)$ (see Figure~\ref{quasihirzecone})
. The dual basis of $v_2,v_4$ is $\alpha_2^\eta=-e_1^*,\alpha_4^\eta=ae_1^*-e_2^*$, so the affine toric quasifold corresponding to the triple $(\eta, \Z\times(\Z+a\Z), \{v_2, v_4\})$
is given by $X_\eta=\mathbb{C}^2/\Gamma_\eta$, where
$$\Gamma_\eta=\{\ [ahv_4] \mid h\in \Z\}\simeq (\mathbb{Z}+a\mathbb{Z})/\Z$$
acts on $\C^2$ as follows
$$[ahv_4]\cdot(z_2,z_4)=(z_2,e^{2\pi i ah}z_4).$$ 
\begin{figure}[ht]
\begin{center}
\includegraphics{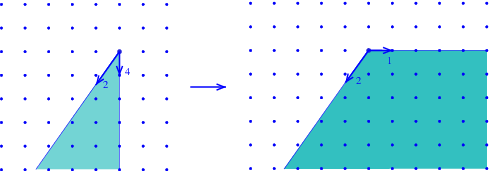}
\caption{The subcone $\eta$ of $\tau$.}
\label{quasihirzecone}
\end{center}
\end{figure}
The identity of $\R^2$
defines a cone morphism
$\eta\rightarrow\tau$. The induced affine quasifold  morphism is given by the surjective map
$$
\begin{array}{cclc}
& X_\eta=\C^2/\Gamma_{\eta}&\longrightarrow & X_\tau=\C^2/\Gamma_{\tau}\\
& [z_2:z_4]&\longmapsto & \left[z_4^{\frac 1a}:z_2z_4^{\frac 1a}\right].
\end{array}
$$}
\end{example}

\subsection{Quasitorus action and equivariance}
Take a triple
$(\delta, Q, \{v^{\delta}_1,\ldots, v^{\delta}_n\})$ with $\delta$ an $r$--dimensional cone. Consider the inclusion $\{0\}\subset\delta$. We have a natural action of the quasitorus 
$X_{\{0\}}^\delta$ on 
$X_\delta$ that is given by
 $$[\exp_\delta(u)]\cdot [z_1:\cdots:z_n]=
     [e^{2\pi i\alpha_1^\delta(u)}z_1:\cdots:e^{2\pi i \alpha_n^\delta(u)}z_n],\quad u\in \C^n.$$
As in the rational case, one can show that this action $X_{\{0\}}^\delta\times X_\delta\rightarrow X_\delta$ is induced by the cone morphism $\{0\}\times\delta\rightarrow\delta$ that takes $(u,v)$ in $u+v$.
Thus $X_{\{0\}}^\delta$ is contained in every affine quasifold $X_\delta$ as an $n$--dimensional orbit. Notice that, as in the classical case, we have a bijective orbit--cone correspondence: $X_\delta$ is the disjoint union of the quasitorus orbits corresponding to the relative interiors of the faces of $\delta$. For example, the orbit corresponding to the $l$--dimensional face $\delta'$ generated by 
$\{v_1^\delta,\ldots,v_l^\delta\}$, with $l\leq r$,
 will be the subset $\{[z_1:\cdots:z_n]\in X_\delta\;|\;z_i\neq0\quad\text{iff}\quad l<i\leq n\}.$
 \begin{example}
\rm{Consider the cone $\delta$ in Example~\ref{cono proiettivo 2}.
The corresponding affine toric quasifold, $(\C\times\C^*)/\Gamma_\tau$,
 is the union of two quasitorus orbits: the full orbit $(\C^*\times\C^*)/\Gamma_\tau$, corresponding to the zero dimensional cone, and the $1$-dimensional orbit $(\{0\}\times \C^*)/\Gamma_\tau$, corresponding the relative interior
 of $\delta$.}
 \end{example}
 Let us now go back to our morphism
 $\Phi_{\delta_2\delta_1}$.
The affine toric morphism
$\Phi_{\{0\}\{0\}}$ sends the quasitorus $X_{\{0\}}^{\delta_1}\subset X_{\delta_1}$ to the quasitorus $X_{\{0\}}^{\delta_2}\subset X_{\delta_2}$ and it is given by $$\Phi_{\{0\}\{0\}}([\exp_{\delta_1}(v)])=[\exp_{\delta_2}(f(v))],\quad v\in\C^n,$$ 
where we take the natural extension of $f$ to $\C^n$. It is therefore the natural group homomorphism induced by $f$ from the first quasitorus to the second.
By Proposition~\ref{restriction}, the 
restriction of
$\Phi_{\delta_2\delta_1}$ to $X_{\{0\}}^{\delta_1}$ it coincides with the homomorphism $\Phi_{\{0\}\{0\}}$ and is therefore naturally equivariant. This extends to all of $X_{\delta_1}$, as shown below:
\begin{prop}\label{equivariance}
The affine toric morphism
$\Phi_{\delta_2\delta_1}$
is equivariant
with respect to the quasitorus actions, namely
$$
\Phi_{\delta_2\delta_1}([\exp_{\delta_1}(v)]\cdot[z_1:\cdots:z_n])=[\exp_{\delta_2}(f(v))]\cdot\Phi_{\delta_2\delta_1}([z_1:\cdots:z_n]), v\in\C^n.$$
\end{prop}
\noindent\begin{preuve}
Take a $[z_1:\cdots:z_n]\in X_{\delta_1}$. Recall that $z_h$ can be $0$ if and only if $h\in I_{\delta_1}$. Now we proceed as in the proof of Theorem~\ref{welldefined} and we replace each  non--zero $z_h$ in (\ref{morfismo}) with
$e^{2\pi i\alpha_h^{\delta_1}(v)}z_h$, $v\in\C^n$.
Then we have
$$\prod_{h=1}^n\left(e^{2\pi i\alpha_h^{\delta_1}(v)}z_h\right)^{\alpha^{\delta_2}_{j}(f(v_h^{\delta_1}))}=e^{2\pi i\alpha^{\delta_2}_{j}(f(v))}\prod_{h=1}^n {z_h}^{\alpha^{\delta_2}_{j}(f(v_h^{\delta_1}))}.
$$ 
This implies that
$$
\begin{array}{cl}
\Phi_{\delta_2\delta_1}([e^{2\pi i\alpha_1^{\delta_1}(v)}:\cdots:e^{2\pi i\alpha_n^{\delta_1}(v)}]\cdot[z_1:\cdots :z_n])=\\

[e^{2\pi i\alpha^{\delta_2}_{1}(f(v))}:\cdots: e^{2\pi i\alpha^{\delta_2}_{m}(f(v))}]\cdot\Phi_{\delta_2\delta_1}[z_1:\cdots :z_n].
\end{array}
$$
\end{preuve}

\subsection{Quasilattice inclusions}\label{quoziente}
Consider a triple
$(\delta,Q,\{v_1^\delta,\ldots,v_n^\delta\})$
and the corresponding affine toric quasifold
$X_\delta$.
Take now a new quasilattice 
$Q'$ containing the quasilattice $Q$.
Consider the new
triple
$(\delta,Q',\{v_1^\delta,\ldots,v_n^\delta\})$
with its own affine toric quasifold
$X'_\delta$. 
The identity of $\R^n$ induces  the equivariant surjective morphism
$$
\begin{array}{cclc}
& X_{\delta}&\longrightarrow & X'_{\delta}\\
& [z_1:\cdots:z_n]&\longmapsto & [z_1:\cdots:z_n].
\end{array}
$$
Given the quotient group
$\Gamma=Q'/Q$, it
is straightforward to check that 
$\Gamma'_\delta/\Gamma_\delta\simeq \Gamma$
and that $\Gamma$ is the
kernel of the restriction of
the above morphism to the torus $(\C^*)^n/\Gamma_{\delta}$.
This construction allows to present the
quotient $X_\delta/\Gamma$ as the affine toric quasifold $X'_\delta$.
\begin{example}
\rm{Take the cone $\sigma$ in 
Examples~\ref{quasispheremodel}, \ref{quasispheremodeltwo}. 
Consider the triples
$(\sigma, \Z, \{1\})$ and $
(\sigma, \Z+a\Z, \{a\}).$
The identity of $\R$ induces the
surjective quasifold morphism of the corresponding affine quasifolds given by
$$
\begin{array}{cclc}
& \C&\longrightarrow & \C/\Gamma_\sigma\\
& z&\longmapsto & \left[z^{\frac 1a}\right].
\end{array}
$$
If $a=n$ is a positive integer, $\Gamma_\sigma=
\Z_n$ and the morphism, as is well--known, is an isomorphism, with inverse given by
$$
\begin{array}{cclc}
& \C/\Z_n&\longrightarrow & \C\\
& [z]&\longmapsto & z^n.
\end{array}
$$}
\end{example}
\section{Algebraic Toric quasifolds}\label{toricquasifolds}
We adapt from \cite[Ch.~3.0]{cox} the general construction of abstract toric varieties obtained by gluing together affine varieties.
\subsection{Definition}
 We consider
a triple $(\Sigma,Q,\{v_1,\ldots,v_d\})$, with $\Sigma\subset\R^n$ a simplicial fan with convex support of full dimension $n$ and we define the corresponding abstract toric quasifold by suitably gluing the following collection of affine toric quasifolds:
$$\left\{X^{\tau}_{\gamma}\;|\;\gamma\subset\tau,\quad \tau\;\text{maximal cone in}\;\Sigma\right\}.$$
Note that when we use the inclusion symbol between two cones, here and 
in the following, it is
understood that the first is a face
of the second and
the corresponding affine toric quasifold
is the one given in Remark~\ref{facce}.
We now describe the gluing maps in detail and prove the compatibility conditions.
 
Let $\delta$ and $\gamma$ be two cones of the fan $\Sigma$ of dimension $r$ and $s$ respectively and let $\sigma$ and $\tau$ be any two maximal cones such that $\delta\subset\sigma$ and $\gamma\subset\tau$. 
Suppose for simplicity that the cone $\sigma$ is generated by the first $n$ vectors, $\{v_1,\ldots,v_n\}$, and
that the cone $\tau$ is generated by
$\{v_{i_1},\ldots,v_{i_n}\}$, with
$\{i_1,\ldots,i_n\}\subset\{1,\ldots,d\}$.
The cone $\delta\cap\gamma$ is a face of $\sigma$ and $\tau$. 
Consider the triples
$
(\delta\cap\gamma,Q,\{v_1,\ldots,v_n\})$ and $
(\delta\cap\gamma,Q,\{v_{i_1},\ldots,v_{i_n}\})
$.
The identity of $\R^n$ defines a cone
isomorphism $\delta\cap\gamma\rightarrow\delta\cap\gamma$. According to (\ref{morfismo}), the corresponding affine quasifold
isomorphism
$h_{\gamma\delta}\colon
X^\sigma_{\delta\cap\gamma}\rightarrow
X^\tau_{\delta\cap\gamma}$	
is given by
$$
h_{\gamma\delta}([ z_1:\cdots:z_n ])= \left[\prod_{k=1}^n z_k^{\alpha^{\tau}_{i_1}(v_{k})}:\cdots:\prod_{k=1}^n z_k^{\alpha^{\tau}_{i_n}(v_k)}\right],$$
where $\alpha^{\tau}_{i_1},\ldots,\alpha^{\tau}_{i_n}$ is the basis that is dual to
the basis $v_{i_1},\ldots,v_{i_n}$ of
ray generators for $\tau$.
The mappings $h_{\gamma\delta}$
are our gluing isomorphisms. 

Note that
$\delta\cap\gamma$ is
also a face of $\sigma\cap\tau$.
By Proposition~\ref{restriction},
$h_{\gamma\delta}$ is the restriction  of $h_{\tau\sigma}$ to $X^\sigma_{\delta\cap\gamma}$.
For this reason, in the examples that will follow, we will only be computing the gluing isomorphisms that correspond to 
the intersection of maximal cones.

Before stating the following theorem, recall from Lemma~\ref{faceinclusion} that there is a natural inclusion of the affine toric quasifolds corresponding to the faces of a cone into the one corresponding to the cone itself. 
\begin{thm}\label{gluing}
For any triple $(\Sigma,Q,\{v_1,\ldots,v_d\})$ 
take 
$$\widetilde{X}_\Sigma=\sqcup_{\delta\subset \sigma} X_{\delta}^{\sigma},$$
where $\delta, \sigma\in\Sigma$ and $\sigma$ is maximal. Take 
$y\in X^{\sigma}_{\delta}$ and $y'\in X^{\tau}_{\gamma}$, with $\delta\subset \sigma$, $\gamma\subset\tau$. Then
 $y\sim y'$, if, and only if, $y\in X^{\sigma}_{\delta\cap\gamma}$, $y'\in X^{\tau}_{\delta\cap\gamma}$ and $y'=h_{\gamma\delta}(y)$. The relation $\sim$ is an equivalence relation.
\end{thm}
\noindent\begin{preuve}
For all face inclusions $\delta\subset\sigma, \gamma\subset\tau, \epsilon\subset\rho$ we have
\begin{itemize}
\item $h_{\gamma\delta}=h_{\gamma\epsilon}\circ h_{\epsilon\delta }$ 
\item $h_{\delta\gamma}=h^{-1}_{\gamma\delta}$
\end{itemize}
by Theorem~\ref{functoriality}. The compatibility conditions above imply that $\sim$ is an equivalence relation. 
\end{preuve}

This allows us to finally make the following fundamental definition:
\begin{dfn}\begin{rm}
    We call the quotient space $X_\Sigma=\widetilde{X}_\Sigma/\sim$ the   \textit{algebraic toric quasifold
    associated with the triple} 
    $(\Sigma,Q,\{v_1,\ldots,v_d\})$.
    We will say that $n$ is the \textit{dimension} of $X_{\Sigma}$.\end{rm}
\end{dfn}
Notice that, for each $\delta\subset\sigma$, the subset $\{[y]\in X_\Sigma\;|\;y\in X^{\sigma}_{\delta}\}$ of $X_\Sigma$ can be identified with $X_{\delta}^{\sigma}$ via the mapping  $h_{\delta}^{\sigma}([y])=y$, thus locally our $X_\Sigma$ can be interpreted as an affine toric quasifold. 

The fact that the gluing isomorphisms are equivariant allows to define an action of the
quasitorus $\C^n/Q$ on the toric quasifold $X_{\Sigma}$. The quasitorus can then be identified with the  orbit of maximal dimension.
Notice that, as in the classical case, the orbit--cone correspondence for affine quasifolds induces an orbit-cone correspondence for the quasifold  $X_\Sigma$.
\begin{rem}[The quotient construction of toric quasifolds]
\rm{As in the rational case \cite{audin,cox0}, algebraic toric quasifolds can also be constructed as complex quotients \cite{cx}. The link between the two constructions is again the atlas of the complex quotient, whose charts are naturally identified with the algebraic affine quasifolds corresponding to the maximal cones.}
\end{rem}
\subsection{Examples}
\begin{example}[The quasisphere]
\rm{
Let $\Sigma$ be the one--dimensional fan in $\R$ given by the simplicial cones $\sigma=\R_{\geq 0}$, $\tau=\R_{\leq 0}$, and $\sigma\cap\tau=\{0\}$.
Take the vectors $v_1=a$, with $a$ a positive real number, and $v_2=-1$. We call the toric quasifold corresponding to
the triple
$$
(\Sigma, \Z+a\Z, \{v_1,v_2\}).
$$
\textit{quasisphere}. It was
first introduced, in the
symplectic setting, in \cite{pcras, p}; a thorough treatment can be found in \cite{p2}.
Notice that, when $a$ equals a positive integer $n$, it equals $\C\P^1$.
We proceed to describe the affine toric quasifolds corresponding to the cones
$\sigma, \tau$ and gluing isomorphisms.
We have treated the triple
$(\sigma, \Z+a\Z, \{v_1\})$
in Example \ref{quasispheremodeltwo}.
So we consider the triple
$(\tau, \Z+a\Z, \{v_2\})$.
The dual basis to $v_2$ is given by the vector $\alpha^{\tau}_2=-e_1^*$. Then 
$$\Gamma_\tau=
(\Z+a\Z)/\Z=\left\{ \left[h a v_2\right]\mid h\in \Z\right\}$$ acts on $\C$ by 
rotations of angle $a$:
$$\left[ahv_2\right]\cdot w=e^{2\pi i ah}w.$$
Thus $X_{\tau}=\C/\Gamma_{\tau}$ and
$X^{\tau}_{\{0\}}=\C^*/\Gamma_\tau$.
The gluing isomorphisms are given
 by $$h_{\tau\sigma}([z])=[z^{\alpha_2^{\tau}(v_1)}]=[z^{-a}]$$ 
 and $$h_{\sigma\tau}([w])=[w^{\alpha_1^{\sigma}(v_2)}]=[w^{-\frac1a}].$$}
\end{example}

\begin{example}[Generalized weighted projective space]
\label{weightedprojective}
\rm{
Consider any positive real number $a$, and take the complete simplicial fan $\Sigma$ whose rays are generated by the vectors  
 $v_1=(1,0)$, $v_2=(-1,-a)$ and $v_3=(0,1)$
 (see Figure~\ref{triangle}).
 \begin{figure}[ht]
\begin{center}
\includegraphics{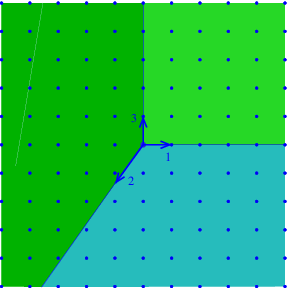}
\caption{The fan for the generalized projective space.}
\label{triangle}
\end{center}
\end{figure}
 The vectors $v_1, v_2, v_3$ span the quasilattice
$$Q=\mathbb{Z}\times (\mathbb{Z}+a\mathbb{Z})\supseteq \mathbb{Z}^2.$$
Consider the triple
$$
(\Sigma, \mathbb{Z}\times (\mathbb{Z}+a\mathbb{Z}), \{v_1, v_2, v_3\}).
$$
We call the corresponding toric quasifold 
\textit{generalized weighted projective space} and we denote it by $\C\P^2_{(1,1,a)}$;
it was first introduced, in the symplectic setting, in \cite{p}. Notice, that for $a=1$ we get $\C\P^2$ while for $a$ equal to a positive integer $n$ we get ordinary weighted projective space $\C\P^2_{(1,1,n)}$.

Let us describe the affine toric quasifolds associated with the maximal cones, together with the corresponding gluing isomorphisms.
Denote by $\sigma$ the maximal cone generated by $v_1, v_3$, by
 $\rho$ the one generated by  $v_2, v_3$, and finally by
 $\tau$ the one generated by $v_1, v_2$ (see Figure~\ref{triangle}). 
Consider the triples:
$$(\sigma, \mathbb{Z}\times (\mathbb{Z}+a\mathbb{Z}), \{v_1, v_3\}),$$
$$(\rho, \mathbb{Z}\times (\mathbb{Z}+a\mathbb{Z}), \{v_2, v_3\}),$$
$$(\tau, \mathbb{Z}\times (\mathbb{Z}+a\mathbb{Z}), \{v_1, v_2\}).$$
Take the first one.
The dual basis of $v_1,v_3$ is $\alpha_1^\sigma=e_1^*,\alpha_3^\sigma=e_2^*$ and
$$\Gamma_\sigma=\{\ [ahv_3] \mid h\in \Z\}
=(\mathbb{Z}\times (\mathbb{Z}+a\mathbb{Z}))/\mathbb{Z}^2 \simeq (\mathbb{Z}+a\mathbb{Z})/\Z.$$
The action of $\Gamma_\sigma$ on $\mathbb{C}^2$ is given by
$$[ahv_3]\cdot(z_1,z_3)=(z_1,e^{2\pi i ah}z_3)$$
and
$$X_\sigma=\mathbb{C}^2/\Gamma_\sigma.$$
Consider now the second triple.
The dual basis of $v_2,v_3$ is $\alpha_2^\rho=-e_1^*,\alpha_3^\rho=-ae_1^*+e_2^*$ and $$\Gamma_\rho=\{\ [ahv_3] \mid h\in \Z\}\simeq (\mathbb{Z}+a\mathbb{Z})/\Z.$$
The action of $\Gamma_\rho$ on $\mathbb{C}^2$ is given by
$$[ahv_3]\cdot(z_2,z_3)=(z_2,e^{2\pi i ah}z_3)$$
and
$$X_\rho=\mathbb{C}^2/\Gamma_\rho.$$
The third triple was already treated in Example~\ref{cono proiettivo 2}. 
Corresponding to the face inclusions
$\sigma\cap\rho\subset\sigma$,
$\sigma\cap\rho\subset\rho$ and relative induced triples, we have $$X_{\sigma\cap\rho}^\sigma=(\mathbb{C}^*\times \mathbb{C})/\Gamma_\sigma,$$
$$X_{\sigma\cap\rho}^\rho=(\mathbb{C}^*\times \mathbb{C})/\Gamma_\rho, $$
$$
h_{\rho\sigma}([z_1:z_3 ])= \left[z_1^{-1}:z_1^{-a}z_3\right],$$
and
$$
h_{\sigma\rho}([z_2:z_3 ])= \left[z_2^{-1}:z_2^{-a}z_3\right].$$
On the other hand, corresponding to the face inclusions
$\sigma\cap\tau\subset\sigma$,
$\sigma\cap\tau\subset\tau$ and relative triples, we have $$X_{\sigma\cap\tau}^\sigma=(\mathbb{C}\times \mathbb{C}^*)/\Gamma_\sigma,$$
$$X_{\sigma\cap\tau}^\tau=(\mathbb{C}\times \mathbb{C}^*)/\Gamma_\tau, $$
$$
h_{\tau\sigma}([z_1:z_3])= \left[z_1z_3^{-\frac{1}{a}}:z_3^{-\frac{1}{a}}\right],$$
and
$$
h_{\sigma\tau}([z_1:z_2 ])= \left[z_1z_2^{-1}:z_2^{-a}\right].$$
One proceeds similarly for the
face inclusions $\rho\cap\tau\subset\rho$,
$\rho\cap\tau\subset\tau$, and gluing isomorphisms.}
\end{example}

\begin{example}[Generalized Hirzebruch surface]
\label{hirzebruch}
\rm{
Consider again any positive real number $a$, and add
the vector $v_4=(0,-1)$ to the vectors 
$v_1=(1,0)$, $v_2=(-1,-a)$, and $v_3=(0,1)$ of the
previous example. 
These four vectors still span the same quasilattice $Q=\mathbb{Z}\times (\mathbb{Z}+a\mathbb{Z})$.
Let now $\Sigma$ be the
complete simplicial fan whose rays are generated by these vectors
(see Figure~\ref{hirze})
\begin{figure}[h]
\begin{center}
\includegraphics{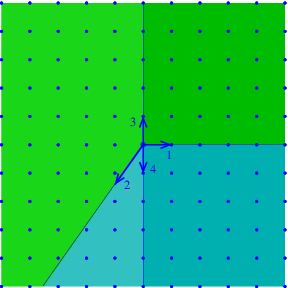}
\caption{The fan for the generalized Hirzebruch surface.}
\label{hirze}
\end{center}
\end{figure}
and consider the triple
$$
(\Sigma, \mathbb{Z}\times (\mathbb{Z}+a\mathbb{Z}), \{v_1, v_2, v_3,v_4\}).
$$
We call the corresponding toric quasifold,
$\H_a$, \textit{generalized Hirzebruch surface}, since, for $a$ equal to a positive integer $n$, this triple yields the standard Hirzebruch surface $\H_n$. 
The toric quasifolds $\H_a$ were first introduced in \cite{hirze} in the symplectic and complex setting, and also from the point of view of foliations. It was also shown there that they form a one--parameter family of toric quasifolds containing the smooth Hirzebruch surfaces $\mathbb{H}_n$.

Take the maximal cones 
$\sigma$, generated by $v_1, v_3$,
$\rho$, generated by  $v_2, v_3$,
$\eta$, generated by  $v_2, v_4$, 
and $\theta$, generated by  $v_1, v_4$.
The affine toric quasifold corresponding to the triples:
$$(\sigma, \mathbb{Z}\times (\mathbb{Z}+a\mathbb{Z}), \{v_1, v_3\}),$$
$$(\rho, \mathbb{Z}\times (\mathbb{Z}+a\mathbb{Z}), \{v_2, v_3\}), $$
and the gluing maps $h_{\rho\sigma}$,
 $h_{\sigma\rho}$
are identical to the ones in the previous
example. Let us focus on the new triples
$$(\eta, \mathbb{Z}\times (\mathbb{Z}+a\mathbb{Z}), \{v_2, v_4\}),$$
$$(\theta, \mathbb{Z}\times (\mathbb{Z}+a\mathbb{Z}), \{v_1, v_4\}).$$
The first was treated in Example~\ref{inclusion}.
For the second, we get that the dual basis of $v_1, v_4$ is 
$\alpha_1^\theta=e_1^*,\alpha_4^\theta=-e_2^*$. The group $$\Gamma_\theta=\{\ [ahv_4] \mid h\in \Z\}\simeq (\mathbb{Z}+a\mathbb{Z})/\Z$$
acts on $\mathbb{C}^2$ by
$$[ahv_4]\cdot(z_1,z_4)=(z_1,e^{2\pi i a h}z_4)$$
and
$$X_\theta=\mathbb{C}^2/\Gamma_\theta.$$
If we now take the face inclusions
$\eta\cap\theta\subset\eta$,
$\eta\cap\theta\subset\theta$, and relative triples, we have
$$X_{\eta\cap\theta}^\eta=(\mathbb{C}^*\times \mathbb{C})/\Gamma_\eta,$$
$$X_{\eta\cap\theta}^\theta=(\mathbb{C}^*\times \mathbb{C})/\Gamma_\theta,$$
$$
h_{\theta\eta}([z_2:z_4])= \left[z_2^{-1}:z_2^a z_4\right],
$$
and
$$
h_{\eta\theta}([z_1:z_4 ])= \left[z_1^{-1}:z_1^a z_4\right].
$$
One proceeds similarly for $\eta\cap\rho$, $\theta\cap\sigma$ and gluing isomorphisms.}
\end{example}
\begin{example}[The Penrose kite]
\rm{
Consider the fifth roots of unity
$$
\begin{array}{l}
u_0=(1,0)\\
u_1=(\cos{\frac{2\pi}{5}},\sin{\frac{2\pi}{5}})=\frac{1}{2}(\frac{1}{\phi},\sqrt{2+\phi})\\
u_2=(\cos{\frac{4\pi}{5}},\sin{\frac{4\pi}{5}})=\frac{1}{2}(-\phi,\frac{1}{\phi}\sqrt{2+\phi})\\
u_3=(\cos{\frac{6\pi}{5}},\sin{\frac{6\pi}{5}})=\frac{1}{2}(-\phi,-\frac{1}{\phi}\sqrt{2+\phi})\\
u_4=(\cos{\frac{8\pi}{5}},\sin{\frac{8\pi}{5}})=\frac{1}{2}(\frac{1}{\phi},-\sqrt{2+\phi}),
\end{array}
$$
where $\phi=\frac{1+\sqrt{5}}{2}$ is the golden ratio, which satisfies the equation
$\phi=1+\frac{1}{\phi}$. These vectors span the pentagonal quasilattice $Q_5$.
Consider now the
complete simplicial fan 
$\Sigma$ whose rays are generated by the vectors $v_1=-u_1$, $v_2=-u_3$, $v_3=u_2$, and $v_4=u_4$
(see Figure \ref{kitefan}). This is the normal fan of the Penrose kite.
 \begin{figure}[ht
 ]
 \begin{center}
\includegraphics{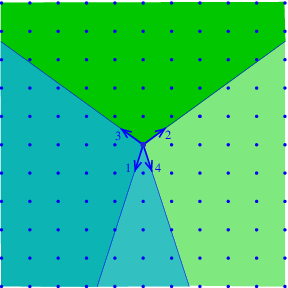}
\caption{The normal fan of the Penrose kite.}
 \label{kitefan}
 \end{center}
\end{figure}
We are interested in the toric quasifold corresponding to the triple
$$
(\Sigma, Q_5, \{v_1, v_2, v_3,v_4\}).
$$
We remark that this toric quasifold was
first introduced and studied in \cite{kite} in the symplectic context. 
As such, it is shown in \cite[Theorem 6.1]{kite} that \textit{it is not} a global quasifold, 
namely not the quotient of a manifold modulo the action of a countable group.

We will focus here on the maximal cone $\sigma$ generated by $v_1, v_4$, the maximal cone $\tau$ generated by $v_2,v_4$, and corresponding gluing maps.
Consider first the
triple $(\sigma, Q_5, \{v_1, v_4\})$. The dual basis of
$v_1, v_4$ is given by $\alpha_1^\sigma=-\phi \,e_1^*-\frac{1}{\sqrt{2+\phi}}\,e_2^*$, $\alpha_4^\sigma=\phi \,e_1^*-\frac{1}{\sqrt{2+\phi}}\,e_2^*$. 
Thus
$$\left\{
\begin{array}{lrrr}
v_2=-&\phi v_1&+& v_4\\
v_3=&v_1&-&\phi v_4,\\
\end{array}
\right.
$$
and we have that
the corresponding affine toric quasifold is given by $$X_\sigma=\mathbb{C}^2/\Gamma_\sigma,$$
with  
$$\Gamma_\sigma=\left\{\ \left[\phi(hv_1+kv_4)\right] \mid h,k \in \Z\right\}$$
acting on $\mathbb{C}^2$ as follows
$$\left[\phi(hv_1+kv_4)\right]\cdot(z_1,z_4)=(e^{2\pi i \phi h} z_1,e^{2\pi i \phi k}z_4).$$
Consider next the
triple $(\tau, Q_5, \{v_2, v_4\})$. The dual basis of
$v_2, v_4$ is given by $\alpha_2^{\tau}=e_1^*+\frac{1}{\phi\sqrt{2+\phi}}\,e_2^*$, $\alpha_4^{\tau}=\frac{1}{\phi}\,e_1^*-\frac{\phi}{\sqrt{2+\phi}}\,e_2^*$.  Therefore
$$\left\{
\begin{array}{lrrr}
v_1=-&\frac{1}{\phi} v_2&+& \frac{1}{\phi}v_4\\
v_3=-&\frac{1}{\phi} v_2&-& v_4,\\
\end{array}
\right.
$$
and we have that 
$$\Gamma_\tau=\left\{\ \left[\frac{1}{\phi}(hv_2+kv_4)\right] \mid h,k \in \Z\right\}$$
The corresponding affine toric quasifold is given by $$X_\tau=\mathbb{C}^2/\Gamma_\tau,$$
with  $\Gamma_{\tau}$
acting on $\mathbb{C}^2$ as follows
$$\left[\frac{1}{\phi}(hv_2+kv_4)\right]\cdot(z_2,z_4)=(e^{2\pi i \frac{h}{\phi}} z_2,e^{2\pi i \frac{k}{\phi}}z_4).$$
Corresponding to the face inclusions
$\sigma\cap\tau\subset\sigma$,
$\sigma\cap\tau\subset\tau$ and relative triples, we have $$X_{\sigma\cap\tau}^\sigma=(\mathbb{C}^*\times \mathbb{C})/\Gamma_\sigma,$$
$$X_{\sigma\cap\tau}^\tau=(\mathbb{C}^*\times \mathbb{C})/\Gamma_\tau, $$
$$
h_{\tau\sigma}([z_1:z_4])= \left[z_1^{-\frac{1}{\phi}}:z_1^{\frac{1}{\phi}}z_4\right],$$
and
$$
h_{\sigma\tau}([z_2:z_4 ])= \left[z_2^{-\phi}:z_2z_4 \right].$$}
\end{example}

\section{Morphisms and blow--ups of algebraic toric quasifolds}\label{morphisms}
We are now ready to define toric morphisms for algebraic toric quasifolds.
All fans are again simplicial with convex support of full dimension.
\subsection{Toric morphisms}
\begin{thm}\label{toricmorphisms}
Let $\Sigma_1\rightarrow \Sigma_2$ be a fan morphism. Then there is an induced map $\Phi\colon X_{\Sigma_1}\rightarrow X_{\Sigma_2}$,
which is equivariant with respect to the
actions of the corresponding quasitori.
\end{thm}
\noindent\begin{preuve}
Consider fan triples $(\Sigma_1,Q_1,\{v_1,\ldots,v_d\})$ and $(\Sigma_2,Q_2,\{u_1,\ldots,u_l\})$ 
for $X_{\Sigma_1}$ and $X_{\Sigma_2}$ respectively.
Take $p\in X_{\Sigma_1}$. Assume that $p=[x]$, for
$x\in X_{\delta_1}^{\sigma_1}$, $\delta_1\subset\sigma_1$. 
By definition, there exist a cone
$\delta_2$ in ${\Sigma_2}$ such that $f(\delta_1)\subset\delta_2$,
where $f\colon\R^n\rightarrow\R^m$ is the linear map that defines the fan morphism. Choose now a maximal cone
$\sigma_2$ in $\Sigma_2$ containing $\delta_2$. Then we define
$$\Phi(p)=[\Phi_{\delta_2\delta_1}(x)],$$
where $\Phi_{\delta_2\delta_1}\colon X_{\delta_1}^{\sigma_1} \rightarrow X_{\delta_2}^{\sigma_2}$ is the affine toric quasifold morphism induced by $f$. We want to show that $\Phi$ is well defined. Assume that $p=[\overline{x}]$, 
$\overline{x}\in X_{\overline{\delta}_1}^{\overline{\sigma}_1}$, for another face inclusion ${\overline{\delta}_1}\subset{\overline{\sigma}_1}$. 
Notice that we have $x\in X_{\delta_1\cap{\overline{\delta}_1}}^{\sigma_1}$,
$\overline{x}\in X_{\delta_1\cap{\overline{\delta}_1}}^{\overline{\sigma}_1}$ and $\overline{x}=h_{\overline{\delta}_1\delta_1}(x)$.
As above, take
${\overline{\delta}_2}$ in ${\Sigma_2}$ such that $f({\overline{\delta}_1})\subset{\overline{\delta}_2}$ and choose a maximal cone
${\overline{\sigma}_2}$ in $\Sigma_2$ containing ${\overline{\delta}_2}$.
We then have the following commutative diagram
$$	
\xymatrix{
\delta_1\cap{\overline{\delta}_1}\ar[r]\ar[d]&\delta_2\cap{\overline{\delta}_2}
\ar[d]\\
\delta_1\cap{\overline{\delta}_1}\ar[r]&\delta_2\cap{\overline{\delta}_2}.
}
$$
The vertical arrows correspond to the identity and the horizontal ones to $f$.
This yields, by Theorem~\ref{functoriality}, the commutative diagram
$$	
\xymatrix{
X_{\delta_1\cap{\overline{\delta}_1}}^{\sigma_1}\ar[r]^{\Phi_{\delta_2\delta_1}}\ar[d]_{h_{\overline{\delta}_1\delta_1}}&X_{\delta_2\cap{\overline{\delta}_2}}^{\sigma_2}
\ar[d]^{h_{\overline{\delta}_2\delta_2}}\\
X_{\delta_1\cap{\overline{\delta}_1}}^{\overline{\sigma}_1}\ar[r]^{\Phi_{\overline{\delta}_2\overline{\delta}_1}}&X_{\delta_2\cap{\overline{\delta}_2}}^{\overline{\sigma}_2}}
$$
which proves that $\Phi$ is well defined.
The mapping $\Phi$ is equivariant since on each affine piece it is given by an affine morphism, which is equivariant.
\end{preuve}
\begin{dfn}[Toric quasifold morphism]{\rm A \textit{morphism} between two toric quasifolds $X_{\Sigma_1}$ and $X_{\Sigma_2}$ is any mapping that is induced by a fan morphism.}
\end{dfn}
\begin{thm}\label{toric functoriality}The composition of two toric quasifold morphisms is the morphism induced by the composition of the corresponding fan morphisms. Moreover, the identity map on an algebraic toric quasifold $X_\Sigma$ is induced by the identity map on the corresponding fan. 
\end{thm}
\noindent\begin{preuve} This is an immediate consequence of Theorem~\ref{functoriality}.
\end{preuve}
See \cite{KLMV} for a view of fan and toric morphisms, and their functoriality, in a different nonrational toric framework. 
\begin{dfn}[Toric quasifold isomorphism] 
\rm{An \textit{isomorphism} between two toric quasifolds $X_{\Sigma_1}$ and $X_{\Sigma_2}$ is any mapping that is induced by a fan isomorphism. Whenever this happens, we will say that the two toric quasifolds are \textit{isomorphic}}.
\end{dfn}
\begin{rem}[Quasilattice inclusion]
\rm{Consider a triple
$(\Sigma,Q,\{v_1,\ldots,v_d\})$
and the corresponding toric quasifold
$X_\Sigma$. Take a new quasilattice
$Q'$ containing $Q$
and the toric quasifold
$X'_\Sigma$ corresponding to $(\Sigma,Q',\{v_1,\ldots,v_d\})$.
Let $\Gamma$ be the countable group
$Q'/Q$.
Arguing for each cone of $\Sigma$ as we did in
Remark~\ref{quoziente}, we get a surjective equivariant morphism
$$F\colon X_\Sigma\rightarrow X'_\Sigma,$$
allowing us to present the quotient
$X_\Sigma/\Gamma$ as $X'_\Sigma$.}
\end{rem}
\subsection{Blow--ups}
As in the rational case, fan subdivisions naturally induce morphisms of the corresponding toric quasifolds. Of course now we have more freedom and we can even subdivide a rational fan in an arbitrary direction, using the remark above (see also
\cite[p.~15]{cut}).

An interesting special case of morphisms that are induced by subdivision is provided by blow--ups. We illustrate the construction with a crucial example.

It is well known classically that, for $n$ a positive integer, $\H_n$ is the blow--up of $\C\P^2_{(1,1,n)}$ at its singular point. It will be of no surprise to the reader that
the same holds, in this nonrational framework, when $n$ is replaced by any positive real number $a$ (see \cite{hirze} for a discussion of this in the symplectic and complex context).

\begin{example}[Blowing--up generalized weighted projective space]
\rm{
Notice that the fan in Example~\ref{hirzebruch} was obtained from the one in Example~\ref{weightedprojective} by subdividing the cone $\tau$ along the ray generated by $v_1+v_2$, yielding
the two cones $\eta$, $\theta$ (see Figures~\ref{triangle}, \ref{hirze}). 
The corresponding toric morphism is the identity on all affine
quasifolds, except for the ones induced by the cone morphisms
$\eta\rightarrow \tau$, 
$\theta\rightarrow\tau$,
plus the ones corresponding to all the faces of $\eta$ and $\theta$.
The affine quasifold morphism corresponding
to the first was already described in Example~\ref{inclusion}.
It is easily checked that the second is given by
$$
\begin{array}{cclc}
& X_\theta=\C^2/\Gamma_{\theta}&\longrightarrow & X_\tau=\C^2/\Gamma_{\tau}\\
& [z_1:z_4]&\longmapsto & \left[z_1z_4^{\frac 1a}:z_4^{\frac 1a}\right].
\end{array}
$$
By Proposition~\ref{restriction}, the morphisms corresponding to the faces are obtained from these by restriction. 
Notice that the subsets $\{[z_1:z_4]\in X_{\eta}\;|\; z_4=0\}$ and $\{[z_1:z_4]\in X_{\theta}\;|\; z_4=0\}$ are mapped to $[0:0]\in X_\tau$. Each subset is what classically would be called the divisor corresponding to the ray generated by $v_4$. Notice that it is $\C\P^1$, as in the smooth case.}
\end{example}
\subsection*{Data availability}
No data was used for the research described in the article.

\end{document}